\documentclass{amsart}
\usepackage[english]{babel}
\usepackage{amsrefs,amssymb,amsmath,amsthm,amsfonts,epsfig,graphicx,psfrag}
\usepackage[all,cmtip]{xy}
\usepackage[utf8]{inputenc}

\theoremstyle{plain}
\newtheorem{thm}{Theorem}[section]
\newtheorem{cor}[thm]{Corolary}
\newtheorem{lem}[thm]{Lemma}
\newtheorem{prop}[thm]{Proposition}
\newtheorem{rmk}[thm]{Remark}

\theoremstyle{definition}
\newtheorem{df}[thm]{Definition}

\theoremstyle{remark}

\DeclareMathOperator{\Emb}{Emb}
\DeclareMathOperator{\diam}{diam}

\DeclareMathOperator{\clos}{clos}

\DeclareMathOperator{\dist}{dist}

\newcommand{\per}{\hbox{Per}}

\newcommand{\R}{\mathbb R}

\newcommand{\Z}{\mathbb Z}

\renewcommand{\epsilon}{\varepsilon}

\begin{document}

\author[A. Artigue]{Alfonso Artigue}
\email{artigue@unorte.edu.uy}
\address{DMEL, Universidad de la Rep\'ublica, Uruguay}
\title{Robustly N-expansive surface diffeomorphisms}
\date{\today}

\begin{abstract}
We give sufficient conditions for a diffeomorphism of a compact surface 
to be robustly $N$-expansive and cw-expansive in the $C^r$-topology. 
We give examples on the genus two surface showing that they need not to be Anosov diffeomorphisms.
The examples are axiom A diffeomorphisms with tangencies at wandering points.
\end{abstract}
\maketitle

\section{Introduction}
Let $M$ be a smooth compact manifold without boundary and 
consider a $C^1$-diffeomorphism $f\colon M\to M$. 
We say that $f$ is \emph{expansive} if 
there is a positive constant $\delta$ such that 
if $x,y\in M$ and $x\neq y$ then there is $n\in\Z$ such that 
$\dist(f^n(x)f^n(y))>\delta$, where $\dist$ is a metric induced by a Finsler $\|\cdot\|$ on the tangent bundle $TM$. 
We say that $f$ is $C^r$-\emph{robustly expansive} if 
it is in the interior of the set of expansive $C^r$-diffeomorphisms. 
In \cite{Ma75} Mañé proved that $f$ is $C^1$-robustly expansive 
if and only if it a \emph{quasi-Anosov} diffeomorphism, i.e.,  
for every tangent vector $v\in TM$, $v\neq 0$, 
the set $\{\|df^n(v)\|\}_{n\in\Z}$ is unbounded.
Also, he proved that $f$ is quasi-Anosov if and only if it 
satisfies Smale's Axiom A and the \emph{quasi-transversality} condition of stable and unstable manifolds: 
$T_xW^s(x)\cap T_xW^u(x)=0$ for all $x\in M$. 
If $M$ is a compact surface then every quasi-Anosov diffeomorphism is Anosov. 
In higher dimensional manifolds there are examples of quasi-Anosov diffeomorphisms not being Anosov, 
see for example \cite{FR}. 
Obviously, every quasi-Anosov $C^r$-diffeomorphism is $C^r$-robustly expansive. 
To our best knowledge it is unknown whether the converse is true for $r\geq 2$.

The results of \cite{Ma75} were extended in several directions. 
In \cite{ArL} Lipschits perturbations of expansive homeomorphisms 
with respect to a hyperbolic metric were considered. 
There it is shown that quasi-Anosov diffeomorphisms are robustly expansive even allowing Lipschitz perturbations.
In \cite{MSS} it is shown that a vector field is $C^1$-robustly expansive 
in the sense of Bowen and Walters \cite{BW} if and only if it is a quasi-Anosov vector field. 
In \cite{ArK} this result is proved for kinematic expansive flows. 
For vector fields with singular (equilibrium) points 
Komuro \cite{K} introduced a definition called $k^*$-expansivity. 
He proved that the Lorenz attractor is $k^*$-expansive, consequently, 
we have a robustly $k^*$-expansive attractor. 
The question of determining whether a compact boundaryless three-dimensional manifold 
admits a $C^1$-robustly $k^*$-expansive vector field with singular points
seems to be an open problem. 

In the discrete-time case the definition of expansivity has several variations. 
Let us start mentioning the weakest one that will be considered in this paper.
We say that $f\colon M\to M$ is a \emph{cw-expansive} diffeomorphism \cite{Ka1}
if there is $\delta>0$ such that if $C\subset M$ is a non-trivial (not a singleton) connected 
set then there is $n\in\Z$ such that $\diam(f^n(C))>\delta$, where $\diam(C)=\sup_{x,y\in C}\dist(x,y)$.
In \cite{Sa97} it is proved that every $C^1$-robustly cw-expansive diffeomorphism is quasi-Anosov. 
In this paper, Section \ref{secEx}, we show that there are $C^2$-robustly cw-expansive 
surface diffeomorphisms that are not quasi-Anosov. For this purpose we introduce 
the notion $Q^r$-Anosov diffeomorphism. 
The idea is to control the order of the tangencies of stable and unstable manifolds as will be explained 
in Section \ref{secQrA}.

In \cites{Mo1,Mo2} (see also \cite{MoSi}) Morales introduced other forms of expansivity 
that will be explained now. 
The first one saids that $f$ is $N$-\emph{expansive} \cite{Mo1}, for a positive integer $N$, 
if there is $\delta>0$ such that if $\diam(f^n(A))<\delta$ for all $n\in\Z$ and some $A\subset M$ 
then $|A|\leq N$, where $|A|$ stands for the cardinality of $A$. 
In this case we say that $\delta$ is an $N$-\emph{expansivity constant}.
In \cite{LZ} examples are given on compact metric spaces 
showing that $N+1$-expansivity does not imply $N$-expansivity 
for all $N\geq 1$. This extends previous results of \cite{Mo1}. 
The examples of Section \ref{secEx} of the present paper are Axiom A diffeomorphisms 
of a compact surface exhibiting this phenomenon. 

From a \emph{probabilistic} viewpoint expansivity can be defined as follows.
For $\delta>0$, $x\in M$  and a diffeomorphism $f\colon M\to M$ consider the set
\[
 \Gamma_\delta(x)=\{y\in M:\dist(f^n(x),f^n(y))\leq\delta\hbox{ for all }n\in\Z\}.
\]
Given a Borel probability measure $\mu$ on
$M$ we say that $f$ is $\mu$-\emph{expansive} \cite{Mo2} if there is $\delta>0$
such that for all $x\in M$ it holds that $\mu(\Gamma_\delta(x))=0$.
We say that $f$ is
\emph{measure-expansive} if it is $\mu$-expansive for every
non-atomic Borel probability measure $\mu$. 
Recall that $\mu$ is non-atomic if $\mu(\{x\})=0$ for all $x\in M$. 
Also, we say that $f$ is \emph{countably-expansive} if there is $\delta>0$ such that for all
$x\in M$ the set $\Gamma_\delta(x)$ is countable.
In \cite{AD} it is shown that, in the general context of compact metric spaces,
countably-expansivity is equivalent to
measure-expansivity.
In Table \ref{tableExp} we summarize these definitions. 
The implications indicated by the arrows are easy to prove. 
\begin{table}[h]
\label{tabla}
\[
\begin{array}{c}
\hbox{expansive} \Leftrightarrow  \hbox{1-expansive} \\
\Downarrow\\ 
\hbox{2-expansive}\\
\Downarrow \\
\hbox{3-expansive}\\
\Downarrow \\
\hbox{\dots}\\
\Downarrow \\
\hbox{N-expansive}\\
\Downarrow \\
\hbox{countably-expansive}\Leftrightarrow\hbox{measure-expansive}\\
\Downarrow \\
\hbox{cw-expansive}\\
\end{array}
\]
\caption{Hierarchy of some generalizations of expansivity.}
\label{tableExp}
\end{table}

As we said before, in \cite{Sa97} it is shown that $C^1$-robustly cw-expansive diffeomorphisms are quasi-Anosov. 
Therefore, in the $C^1$-category, all the definitions of Table \ref{tableExp} coincide in the robust sense.
The purpose of the present paper is to investigate robust expansivity and its generalizations 
in the $C^r$-topology for $r\geq 2$. 
Let us explain our results while describing the contents of the article.
In Section \ref{secOmegaExp} we recall some known results while introducing the notion of $\Omega$-\emph{expansivity}, 
i.e., expansivity in the non-wandering set. 
We prove that 
$C^1$-robust $\Omega$-expansivity is equivalent with $\Omega$-stability.
In Section \ref{secQrA} we introduce $Q^r$-Anosov $C^r$-diffeomorphisms of compact surfaces. 
In Corollary \ref{corQrA} we show that $Q^r$-Anosov diffeomorphisms are $C^r$-robustly $r$-expansive (i.e., $N$-expansive with $N=r$).
In Section \ref{secPer} we investigate the converse of Corollary \ref{corQrA}.
We show that every periodic point of a $C^r$-robustly 
cw-expansive diffeomorphism on a compact surface is hyperbolic. 
Also, we prove that if $f$ is an Axiom A diffeomorphism without cycles and $C^r$-robustly 
cw-expansive then $f$ is $Q^r$-Anosov.
In Section \ref{secEx} we prove that for each $r\geq 2$ there is a $C^r$-robustly $r$-expansive 
surface diffeomorphism that is not $(r-1)$-expansive.

I am in debt with J. Brum and R. Potrie for conversations related with the examples presented in Section \ref{secEx}.
Some of the results of this paper are part of my Thesis made under the guidance of M. J. Pacifico and J. L. Vieitez. 

\section{Omega-expansivity}
\label{secOmegaExp}
Let $M$ be a smooth compact manifold without boundary. 
In this section the dimension of $M$ will be assumed to be greater than one.
Given a $C^1$-diffeomorphism $f\colon M\to M$ define 
$\per(f)$ as the set of periodic points of $f$ and the non-wandering set 
$\Omega(f)$ as the set of those $x\in M$ satisfying: 
for all $\epsilon>0$ there is $n\geq 1$ such that $B_\epsilon(x)\cap f^n(B_\epsilon(x))\neq\emptyset$.
Recall that $f$ satisfies Smale's \emph{Axiom A} 
if $\clos(\per(f))=\Omega(f)$ and $\Omega(f)$ is hyperbolic. 
A compact invariant set $\Lambda\subset M$ is \emph{hyperbolic} 
if 
the tangent bundle over $\Lambda$ 
splits as $T_{\Lambda}M=E^s\oplus E^u$ 
the sum of two sub-bundles invariant by $df$
and there are $c>0$ and $\lambda\in (0,1)$ such that:
\begin{enumerate}
 \item if $v\in E^s$ then $\|df^n(v)\|\leq c\lambda^n \|v\|$ for all $n\geq 0$ and 
 \item if $v\in E^u$ then $\|df^n(v)\|\leq c\lambda^n \|v\|$ for all $n\leq 0$.
\end{enumerate}

\begin{df}
\index{homeomorphism!$\Omega$-expansive} We say that $f$ is $\Omega$-\emph{expansive} if $f\colon \Omega(f)\to\Omega(f)$ is expansive. 
 We say that $f$ is $C^r$-\emph{robustly} $\Omega$-\emph{expansive} if there is a 
 $C^r$-neighborhood $U$ of $f$ such that every $g\in U$ is $\Omega$-expansive.
\end{df}

\begin{rmk}
Trivially, if $\Omega(f)$ is a finite set then $f$ is $\Omega$-expansive. 
Also, if $\Omega(f)$ is a hyperbolic set then $f$ is $\Omega$-expansive. 
\end{rmk}

A $C^1$-diffeomorphism $f\colon M\to M$ is $\Omega$-\emph{stable} if there is 
a $C^1$-neighborhood $U$ of $f$ such that 
for all $g\in U$ there is a homeomorphism $h\colon \Omega(f)\to\Omega(g)$ 
such that $h\circ f=g\circ h$. 
We say that $f$ is a $C^r$-\emph{star diffeomorphism} if there is a $C^r$-neighborhood $U$ of $f$ such that 
every periodic orbit of every $g\in U$ is a hyperbolic set.
If $f$ satisfies the axiom A then $\Omega(f)$ decomposes into a finite 
disjoint union basic sets $\Omega(f)=\Lambda_1\cup\dots\cup\Lambda_l$. 
A collection $\Lambda_{i_1},\dots,\Lambda_{i_k}$ is called a \emph{cycle} 
if there exist points $a_j\notin\Omega(f)$, for $j=1,\dots,k$, 
such that $\alpha(a_j)\subset \Lambda_{i_j}$ 
and $\omega(a_j)\subset \Lambda_{i_{j+1}}$ (with $k+1\equiv 1$). 
We say that $f$ \emph{has not cycles} (or satisfies the no cycle condition)
if there are not cycles among the basic sets of $\Omega(f)$. 
See for example \cites{Robinson,Sh87} for more on this subject.

%

From \cites{Sm70,Fr71,Ao92,Ha} we know that the following statements are equivalent in the $C^1$-topology:
%
 \begin{enumerate}
  \item $f$ satisfies axiom A and has not cycles,
  \item $f$ is $\Omega$-stable,
  \item $f$ is a star diffeomorphism.
 \end{enumerate}
%
%
%
%



We add another equivalent statement, with a \emph{simple} proof based on deep results, 
that will be used in the next sections.

\begin{prop}
\label{propOmExp}
A diffeomorphism is $C^1$-robustly $\Omega$-expansive if and only if it is $\Omega$-stable.
\end{prop}

\begin{proof}
 In order to prove the direct part suppose that $f$ is $C^1$-robustly $\Omega$-expansive. 
 If a periodic point of $f$ is not hyperbolic then, by \cite{Fr71}*{Lemma 1.1}, 
 we find a small $C^1$-perturbation of $f$ with an arc of periodic points. 
 This contradicts the $C^1$-robust expansivity of $f$ 
 because every expansive homeomorphism of a compact metric space has at most a countable set of periodic points.
 This proves that $f$ is a star diffeomorphism.

 If $f$ is $\Omega$-stable then $f$ satisfies Smale's axiom A. 
 Therefore $\Omega(f)$ is hyperbolic and consequently $f\colon\Omega(f)\to\Omega(f)$ is expansive. 
 Since $f$ is $\Omega$-stable we have that $f$ is robustly $\Omega$-expansive. 
\end{proof}

%
%
%
%
%
%
%
%
%

\section{$Q^r$-Anosov diffeomorphisms}
\label{secQrA}
In this section we assume that $M=S$ is a compact surface, i.e., $\dim(M)=2$.
The \emph{stable set} of $x\in S$ is
\[
 W^s_f(x)=\{y\in S:\lim_{n\to+\infty}\dist(f^n(x),f^n(y))= 0\}.
\]
The \emph{unstable set} is defined by $W^u_f(x)=W^s_{f^{-1}}(x)$.
Assume that $f$ is $\Omega$-stable, $E^s, E^u$ are one-dimensional
and define $I=[-1,1]$.
We denote by $\Emb^r(I,S)$ the space of $C^r$-embeddings of $I$ in $S$ with the $C^r$-topology.
Let us recall the following fundamental result for future reference.

\begin{thm}[Stable manifold theorem]
\label{teoSMT}
 Let $\Lambda\subset S$ be a hyperbolic set of a 
 $C^r$-diffeomorphisms $f$ of a compact surface $S$. 
 Then, for all $x\in\Lambda$, $W^s_f(x)$ is an injectively immersed $C^r$-submanifold. 
 Also the map $x\mapsto W^s_f(x)$ is continuous: 
 there is a continuous function $\Phi\colon\Lambda\to\Emb^r(I,S)$ 
 such that for each
 $x\in\Lambda$ it holds that the image of $\Phi(x)$ is a neighborhood of $x$ in $W^s_f(x)$. 
 Finally, these stable manifolds also depend continuously on the diffeomorphisms $f$, 
 in the sense that nearby diffeomorphisms yield nearby mappings $\Phi$.
 \end{thm}
\begin{proof}
 See  \cite{PaTa} Appendix 1.
\end{proof}

Given $x\in S$ we can take 
$\delta_1,\delta_2>0$, a $C^r$-coordinate chart 
$\varphi\colon U\subset S\to [-\delta_1,\delta_1]\times[\delta_2,\delta_2]$ 
such that $\varphi(x)=(0,0)$ and 
two $C^r$ functions $g^s,g^u\colon [-\delta_1,\delta_1]\to[\delta_2,\delta_2]$ such that 
the graph of $g^s$ and $g^u$ are the local 
expressions of the local stable and the local unstable manifold of $x$, respectively. 
If the degree $r$ Taylor polynomials of $g^s$ and $g^u$ at $0$ 
coincide we say that there is an $r$-\emph{tangency} at $x$.


\begin{df}
Given $r\geq 1$, a $C^r$-diffeomorphisms $f\colon S\to S$ is \emph{$Q^r$-Anosov} 
if it is 
axiom A, has no cycles and
there are no $r$-tangencies.
%
\end{df}

\begin{rmk}
For $r=1$ we have that $Q^1$-Anosov is quasi-Anosov, and in fact, given that $S$ is two-dimensional, it is Anosov. 
 For $r=2$ we are requiring that if there is a tangency of a stable and an unstable manifold it is a quadratic one.
\end{rmk}

We will show that $Q^r$-Anosov diffeomorphisms form an open set of $N$-expansive diffeomorphisms. 
Several results from \cite{Sh87} will be used.
%
%
%
For $\delta>0$ define
\[
  \begin{array}{l}
    W^s_\delta(x,f)=\{y\in S:\dist(f^n(x),f^n(y))\leq\delta\,\,\forall n\geq 0\},\\
    W^u_\delta(x,f)=\{y\in S:\dist(f^n(x),f^n(y))\leq\delta\,\,\forall n\leq 0\}.
  \end{array}
\]
\begin{thm}
\label{thmQrAbierto}
 In the $C^r$-topology 
 the set of $Q^r$-Anosov diffeomorphisms of a compact surface 
 is a $C^r$-open set. 
\end{thm}

\begin{proof}
We know that the set of axiom A $C^r$-diffeomorphisms without cycles 
form an open set $U$ in the $C^r$-topology.
Let $g_k$ be a sequence in $U$
converging to $f\in U$. 
Assume that $g_k$ is not $Q^r$-Anosov for all $k\geq 0$. 
In order to finish the proof is it sufficient to show that $f$ is not $Q^r$-Anosov. 
Since $g_k\in U$ and it is not $Q^r$-Anosov, there is $x_k\in S$ with an $r$-tangency for $g_k$. 

By \cite{Sh87}*{Proposition 8.11} we know that $\Omega(f)$ has a local product structure, 
then, we can apply \cite{Sh87}*{Proposition 8.22} to conclude that $\Omega(f)$ is uniformly locally maximal, 
that is, there are neighborhoods $U_1\subset S$ of $\Omega(f)$ and $U_2$ a $C^r$-neighborhood of $f$ 
such that $\Omega(g)=\cap_{n\in\Z}\,g^n(U_1)$ for all $g\in U_2$. 
Consider the compact set $K=S\setminus U_1$. 
We have that for all $x\notin \Omega(g)$, with $g\in U_2$, 
there is $n\in\Z$ such that $g^n(x)\in K$.
Notice that if $g_k$ has an $r$-tangency at $x_k$ then every point in its orbit 
by $g_k$ has an $r$-tangency too. 
Therefore we can assume that $x_k\in K$ for all $k\geq 1$. 
Since $K$ is compact we can suppose that $x_k\to x\in K$. 
By \cite{Sh87}*{Proposition 9.1} we can take 
$y_k,z_k\in \Omega(g_k)$ such that $x_k\in W^s_{g_k}(y_k)\cap W^u_{g_k}(z_k)$. 
Suppose that $y_k\to y$ and $z_k\to z$ with $y,z\in \Omega(f)$ \cite{Sh87}*{Theorem 8.3}.
By Theorem \ref{teoSMT} we know that 
for some $\delta>0$ the local manifolds $W^s_\delta(y_k,g_k)$ and $W^u_\delta(z_k,g_k)$ 
converges in the $C^r$-topology to $W^s_\delta(y,f)$ and $W^u_\delta(z,f)$ respectivelly. 
Since $K$ is compact and disjoint from $\Omega(f)$ there is $m>0$ such that $$x_k\in g_k^{-m}(W^s_\delta(y_k,g_k))\cap g_k^m(W^u_\delta(z_k,g_k))$$ 
for all $k\geq 1$. 
Then, taking limit $k\to \infty$ we find an $r$-tangency at $x$ for $f$. 
Therefore $f$ is not $Q^r$-Anosov. 
This proves that the set of $Q^r$-Anosov $C^r$-diffeomorphisms is an open set in the $C^r$-topology.
\end{proof}

\begin{df}
  We say that a $C^r$-diffeomorphism $f$ is $C^r$-\emph{robustly} $N$-expansive if 
 there is a $C^r$-neighborhood of $f$ such that every diffeomorphism in this neighborhood is $N$-expansive.
\end{df}


The following is an elementary result from Analysis.

\begin{lem}
\label{lemRolle}
 If $g\colon \R\to\R$ is a $C^r$ functions with $r+1$ roots in the interval 
 $[a,b]\subset\R$ then $g^{(n)}$ has $r+1-n$ roots in $[a,b]$ for all $n=1,2,\dots,r$ where $g^{(n)}$ stands for the $n^{th}$ derivative of $g$.
 \end{lem}

\begin{proof}
 It follows by induction in $n$ using 
 the Rolle's theorem.
\end{proof}

Recall that $r$-expansivity means $N$-expansivity with $N=r$. 

\begin{thm}
\label{teoRobNexp}
Every $Q^r$-Anosov diffeomorphism of a compact surface is 
$r$-expansive. 
Moreover, if $f$ is $Q^r$-Anosov then there are a $C^r$ neighborhood $\mathbb{U}$ of $f$ and $\delta>0$ 
such that $\delta$ is an $r$-expansive constant for every $g\in \mathbb{U}$.
\end{thm}

\begin{proof}
Let $f\colon S\to S$ be a $Q^r$-Anosov diffeomorphism. 
We can take a neighborhood $U$ of $\Omega(f)$, a $C^r$-neighborhood $\mathbb{U}'$ of $f$ 
and $\delta'>0$ such that:
\begin{enumerate}
\item $\Omega (g)$ is expansive with expansivity constant $\delta'$ for all $g\in\mathbb{U}'$ and
\item $\Omega(g)=\cap_{n\in\Z}\, g^n(U)$ for all $g\in\mathbb{U}'$.
\end{enumerate}
Therefore, we have to show that there is $\delta>0$ such that if $X\cap \Omega (f)=\emptyset$
and $\diam (f^n(X))<\delta$ for all $n\in \Z$ then $|X|\leq r$.
Arguing by contradiction assume that there are $g_n$ converging to $f$ in the $C^r$-topology and 
two sequences $s_n$ and $u_n$ of arcs in $S$ 
such that $s_n$ is stable for $g_n$, $u_n$ is unstable for $g_n$, $|s_n\cap u_n|>r$ and $\diam(s_n),\diam(u_n)\to 0$ as $n\to+\infty$. 
Considering the compact set $K$ of the proof of Theorem \ref{thmQrAbierto}
we can assume that $s_n,u_n\subset K$.
Take $x\in K$ an accumulation point of $s_n$. 
By Lemma \ref{lemRolle} and the arguments in the proof of Theorem \ref{thmQrAbierto} 
we have that there is an $r$-tangency at $x$ for $f$. 
This contradiction finishes the proof.
\end{proof}

From Theorems \ref{thmQrAbierto} and \ref{teoRobNexp} we deduce:

\begin{cor}
\label{corQrA}
Every $Q^r$-Anosov diffeomorphism of a compact surface is 
$C^r$-robustly $r$-expansive with uniform $r$-expansivity constant on a $C^r$-neighborhood.
\end{cor}

\section{Robust cw-expansivity} 
\label{secPer}
In this section we will prove that if 
$f$ is $C^r$-robustly cw-expansive, $r\geq 1$, 
then its periodic points are hyperbolic.
It is a first step in the direction of proving the converse of Corollary \ref{corQrA} (in case that this converse is true).
A second step is done in Theorem \ref{thmPertCr} assuming that the diffeomorphism is Axiom A without cycles.

\begin{lem}
\label{lemPert}
Consider a $C^\infty$ manifold $M$ of dimension $n\geq 1$, $p\in M$ and $U\subset M$ 
a neighborhood of $p$. 
Then there are $\epsilon>0$ and a one-parameter family of $C^\infty$ diffeomorphisms 
$f_\mu\colon M\to M$, $|\mu-1|<\epsilon$, 
such that for all $\mu$: $f_\mu(p)=p$, $f_\mu(x)=x$ for all $x\in M\setminus U$, 
$d_pf_\mu=\mu Id$. 
Moreover, for all $r\geq 0$ the function $\mu\mapsto f_\mu$ is continuous in the $C^r$-topology.
\end{lem}

\begin{proof}
Taking a local chart the problem is reduced to Euclidean $\R^n$. 
Then we will assume that $M=\R^n$, $p=0$ and $\clos(B_s(p))\subset U$ for some $s>0$.
Consider a $C^\infty$ function $\rho\colon\R\to[0,1]$ such that 
$\rho(x)=1$ if $x\leq s/2$ and $\rho(x)=0$ if $x\geq s$. 
Define $f_\mu\colon\R^n\to\R^n$ by 
\begin{equation}
\label{eqPerturb}
  f_\mu(x)
  =x+(\mu-1)\rho(\|x\|)x
\end{equation} 
where $\|\cdot\|$ denotes the Euclidean norm.
If $\|x\|<s/2$ then $f_\mu(x)=x+(\mu-1)x=\mu x$. 
Therefore, $d_pf=\mu I$. The rest of the details are direct from (\ref{eqPerturb}).
\end{proof}

A point $x\in M$ is \emph{Lyapunov stable} for $f\colon M\to M$ 
if for all $\epsilon>0$ there is $\delta>0$ such that 
if $\dist(y,x)<\delta$ then $\dist(f^n(x),f^n(y))<\epsilon$ for all $n\geq 0$. 
In \cite{Ka2}*{Theorem 1.6} it is shown that cw-expansive homeomorphisms admits no stable points. 
This is done for a Peano continuum, as is our compact connected manifold $M$. 
This result was previously proved by Lewowicz and Hiraide for expansive homeomorphisms 
on compact manifolds. 
This is a key point in the following proof.

\begin{prop}
\label{propEigen}
If $f$ is a $C^r$-robustly cw-expansive diffeomorphism on a compact manifold $M$ (arbitrary dimension)
and $f^l(p)=p$ then $d_pf^l$ has at least one eigenvalue of modulus greater than 1 and 
at least one eigenvalue of modulus less than 1.
\end{prop}

\begin{proof}
Consider an open set $U\subset M$ containing the periodic point $p$ and such that $f^i(p)\notin \clos(U)$ for all $i=1,\dots,l-1$. 
Arguing by contradiction assume that 
the eigenvalues of $d_pf^l$ are smaller or equal than 1 in modulus (being the other case similar). 
Take from Lemma \ref{lemPert} a $C^r$-diffeomorphism $f_\mu$ of $M$ 
fixing $p$ and being the identity outside $U$. 
In particular, $f_\mu$ is the identity in a neighborhood of the points $f(p),\dots,f^{l-1}(p)$. 
Assume that $\mu\in (0,1)$ is close to 1. 
Define $g=f\circ f_\mu$. 
In this way $p$ is a periodic point of $g$ of period $l$, $g$ is $C^r$-close to $f$ and the 
eigenvalues of $d_pg^l$ are $\mu\lambda_1,\dots,\mu\lambda_j$ wich have modulus (strictly) smaller than 1
(being $\lambda_1,\dots,\lambda_j$ the eigenvalues of $d_pf^l$). 
Then $p$ is a hyperbolic sink for $g$, in particular it is Lyapunov stable. 
Since $f$ is $C^r$-robustly cw-expansive, 
we can assume that $g$ is cw-expansive, arriving to a contradiction with \cite{Ka2}*{Theorem 1.6}. 
\end{proof}

This proposition has the following direct consequence on two-dimensional manifolds:

\begin{cor}
  Every $C^r$-robustly cw-expansive diffeomorphism on a compact surface is a $C^r$-star diffeomorphism.
\end{cor}


To our best knowledge it is not known whether for $r\geq 2$ every $C^r$-star diffeomorphism is Axiom A, 
even for $M$ a compact surface. 
The next result is another partial result in the direction of proving the converse of 
Corollary \ref{corQrA}.

\begin{thm}
\label{thmPertCr}
  Let $f\colon S\to S$ be a $C^r$-diffeomorphism Axiom A without cycles. 
  If $f$ is $C^r$-robustly cw-expansive then $f$ is $Q^r$-Anosov.
\end{thm}

\begin{proof}
We will argue by contradiction assuming that $f$ is not $Q^r$-Anosov. 
This implies that there is a wandering point $p\in S$ with an $r$-tangency. 
Take $C^r$ local coordinates $\phi\colon I\times J\subset\R^2\to S$ around $p$, 
where $I,J\subset \R$ are open intervals.
Since $p$ is a wandering point we can suppose that $f^n(\phi(I\times J))\cap \phi(I\times J)=\emptyset$ for all 
$n\in\Z$, $n\neq 0$.
Let $g_s,g_u\colon I\to J$ be $C^r$ functions such that 
their graphs describe the local stable and local unstable manifold of $p$ in coordinates. 
Since there is an $r$-tangency at $p$ we can suppose that 
the Taylor polynomials of order $r$ of $g_s$ and $g_u$ vanishes at $0$. 

Define the $C^r$ diffeomorphism $h\colon I\times \R\to I\times \R$ as 
$$h(x,y)=(x,g_s(x)-g_u(x)+y).$$ 
Let $\sigma\colon \R\to [0,1]$ be a $C^\infty$ function such that 
$\sigma(a)=1$ for $a\leq 1/2$ and $\sigma(a)=0$ for $a\geq 1$. 
For $\mu>0$ define $j_\mu\colon I\times J\to I\times\R$ as
$$j_\mu(x,y)=\sigma(\sqrt{x^2+y^2}/\mu)h(x,y).$$
Define $f_\mu\colon S\to S$ by
\[
f_\mu(q)=\left\{
\begin{array}{ll}
f\circ \phi\circ j_\mu\circ\phi^{-1}(q) & \hbox{ if } q\in\phi(B_\mu(0,0))\\
f(q) & \hbox{ in other case.}
\end{array}
\right.
\]

Let us show that $f_\mu$ is not cw-expansive for $\mu>0$ small. 
If $\sqrt{x^2+y^2}<\mu/2$ then $j_\mu(x,y)=h(x,y)$. 
Then, we must note that $h(x,g_u(x))=(x,g_s(x))$ 
and this means that $h$ maps the graph of $g_u$ into the graph of $g_s$. 
For $f_\mu$ this implies that there is an arc 
$\gamma=\{\phi(x,g_s(x)): |x|<\mu/\sqrt8\}\subset S$ 
such that $\diam(f^n(\gamma))\to 0$ as $n\to \pm\infty$. 
Then, arbitrarily small subarcs of $\gamma$ contradict the cw-expansivity of each $f_\mu$ for arbitrarily small 
cw-expansive constants. 

Now we will show that $f_\mu$ is a $C^r$ small perturbation $f$ if $\mu$ is close to 0. 
By definition, they coincide for $q\notin\phi(B_\mu(0,0))$. 
Therefore, the problem is reduced to show that $j_\mu$ is a $C^r$ small perturbation of the identity in 
$I\times J$.
Notice that 
$j_\mu(x,y)-Id(x,y)=\sigma(\sqrt{x^2+y^2}/\mu)h(x,y)=(0,g_s(x)-g_u(x))$. 
In order to conclude we will show that the map 
$$l(x,y)=\sigma(\sqrt{x^2+y^2}/\mu)(0,g_s(x)-g_u(x))$$ is $C^r$-close to $(x,y)\mapsto (0,0)$. 
Because of the $r$-tangency at $p$ we know that $R(x)=g_s(x)-g_u(x)$ satisfies 
$R(x)/x^r\to 0$ as $x\to 0$.
This and L'Hospital's rule implies that 
$R^{(t)}(x)/x^{r-t}\to 0$ as $x\to 0$ for all $t=0,1,\dots,r$. 
As before, $R^{(t)}(x)$ denotes the $t^{th}$ derivative of $R$ at $x$.
Define $\rho(x,y)=\sigma(\sqrt{x^2+y^2})$ and let $K=\|\rho\|_{C^r}$. 
Given $\epsilon>0$ consider $\mu>0$ such that 
$$\left|\frac{R^{(t)}(x)}{x^{r-t}}\right|<\frac\epsilon{r!rK}$$ for all $x\in (-\mu,\mu)$.
For $\|(x,y)\|\geq\mu$ we have that $\rho(x/\mu,y/\mu)=0$, so there is nothing to estimate. 
We will assume that $\|(x,y)\|\leq\mu$. 
Given two non-negative integers $i,j$ such that $i+j\leq r$ we have that 
\[
  \begin{array}{ll}
    \displaystyle\left|\frac{\partial^i}{\partial x^i}\frac{\partial^j}{\partial y^j}[\rho(x/\mu,y/\mu)R(x)]\right|
    &=\displaystyle
      \left|
	\frac{\partial^i}{\partial x^i} \left[\frac{\partial^j\rho}{\partial y^j}(x/\mu,y/\mu)\frac1{\mu^j}R(x)\right]
      \right|\\
    &\leq\displaystyle
      \frac1{\mu^j}
	\left|
	  \sum_{l=0}^i i!\frac{\partial^l}{\partial x^l}\frac{\partial^j\rho}{\partial y^j}(x/\mu,y/\mu)
	  \frac1{\mu^l}
	  R^{(i-l)}(x)
	\right|\\
    &\leq\displaystyle 
      \frac{r!}{\mu^{j+l}}\|\rho\|_{C^r}\left|\sum_{l=0}^i R^{(i-l)}(x) \right|.\\
    &\leq\displaystyle 
      r!K\left|\sum_{l=0}^i \frac{R^{(i-l)}(x)}{x^{j+l}} \right|
      \leq r!K\sum_{l=0}^i \left|\frac{R^{(i-l)}(x)}{x^{r-i+l}} \right|\leq \epsilon.\\
      
  \end{array}
\]
This proves that $f_\mu$ is a $C^r$-approximation of $f$ that is not cw-expansive. 
This contradiction proves the theorem.
\end{proof}

\section{Examples of $N$-expansive diffeomorphisms}
\label{secEx}
In this section we present examples of $C^r$-robustly $N$-expansive surface diffeomorphisms 
that are not Anosov. They are variations of the 2-expansive homeomorphism presented in \cite{APV}, 
that in turn, is based on the three-dimensional quasi-Anosov diffeomorphism given in \cite{FR}.

\begin{thm}
 \label{TeoB}
For each $r\geq 2$ there is a $C^r$-robustly $r$-expansive 
surface diffeomorphism that is not $(r-1)$-expansive.
\end{thm}

\begin{proof}
We start with the case of $r=2$. 
It is essentially the example given in \cite{APV}, we recall some details from this paper.
Consider $S_1$ and $S_2$ two copies of the torus $\R^2/\Z^2$
and the $C^\infty$-diffeomorphisms $f_i\colon S_i\to S_i$, $i=1,2$, such that:
1) $f_1$ is a derived-from-Anosov as detailed in \cite{Robinson},
2) $f_2$ is conjugate to $f_1^{-1}$ and
3) $f_i$ has a fixed point $p_i$, where $p_1$ is a source and $p_2$ is a sink.
Also assume that there are local charts $\varphi_i\colon D\to S_i$, $D=\{x\in\R^2:\|x\|\leq 2\}$,
such that
\begin{enumerate}
 \item $\varphi_i(0)=p_i$,
\item the pull-back of the stable (unstable) foliation by $\varphi_1$ ($\varphi_2$)
is the vertical (horizontal) foliation on $D$ and
\item $\varphi_1^{-1}\circ f^{-1}_1\circ \varphi_1(x)=\varphi_2^{-1}\circ f_2\circ \varphi_2(x)=x/4$
for all $x\in D$.
\end{enumerate}

Let $A$ be the annulus $\{x\in\R^2:1/2\leq \|x\|\leq 2\}$.
Consider the diffeomorphism $\psi\colon A\to A$ given by $\psi(x)=x/\|x\|^2$.
Denote by $\hat D$ the open disk $\{x\in\R^2:\|x\|<1/2\}$.
On $[S_1 \setminus \varphi_1(\hat D)]\cup [S_1\setminus \varphi_2(\hat D)]$
consider the equivalence relation generated by
\[
 \varphi_1(x)\sim \varphi_2\circ\psi (x)
\]
for all $x\in A$. 
Denote by $\overline x$ the equivalence class of $x$.
The surface
\[
S= \frac{[S_1 \setminus \varphi_1(\hat D)]\cup [S_1\setminus \varphi_2(\hat D)]}\sim
\]
has genus two, we are considering the quotient topology on $S$.
Define the $C^\infty$-diffeomorphism $f\colon S\to S$ by
\[
 f(\overline x)=\left\{
\begin{array}{ll}
\overline{f_1(x)} &\hbox{ if } x\in S_1 \setminus \varphi_1(\hat D)\\
\overline{f_2(x)} &\hbox{ if } x\in S_2 \setminus \varphi_2(D)\\
\end{array}
\right.
\]

We know that $f$ is Axiom A because
the non-wandering 
set consists of a hyperbolic repeller and a hyperbolic attractor. Also $f$ has no cycles.
The stable and unstable foliations in the annulus $\overline A=\varphi_1(A)$ looks like Figure \ref{fig}.
\begin{figure}[h]
\begin{center}
\includegraphics[scale=.8]{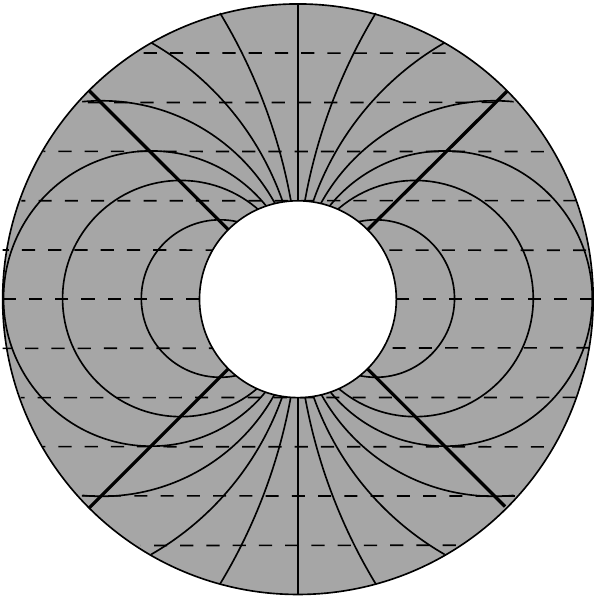}
\caption{Foliations in the annulus $\overline A$.
The circle arcs represent the unstable foliation after the inversion and the horizontal 
dot-lines are the stable foliation. 
The diagonal bold lines are the tangencies between stable and unstable manifolds.}
\label{fig}
\end{center}
\end{figure}
The tangencies are quadratic because in local charts stable manifolds are straight 
lines and unstable manifolds are circle arcs.
Then, applying Theorem \ref{teoRobNexp} we have that $f$ is $C^2$-robustly 2-expansive.
It is not 1-expansive (i.e. expansive) because near the line of tangencies we find pairs of points 
contradicting expansivity (for every expansive constant).

For the case $r=3$ we will change $\psi$ in an open set $U$ contained in $A$. 
Consider $\psi$ such that the stable and the unstable foliations looks like in 
Figure \ref{fig3exp}. 
\begin{figure}[h]
\begin{center}
\includegraphics[scale=.5]{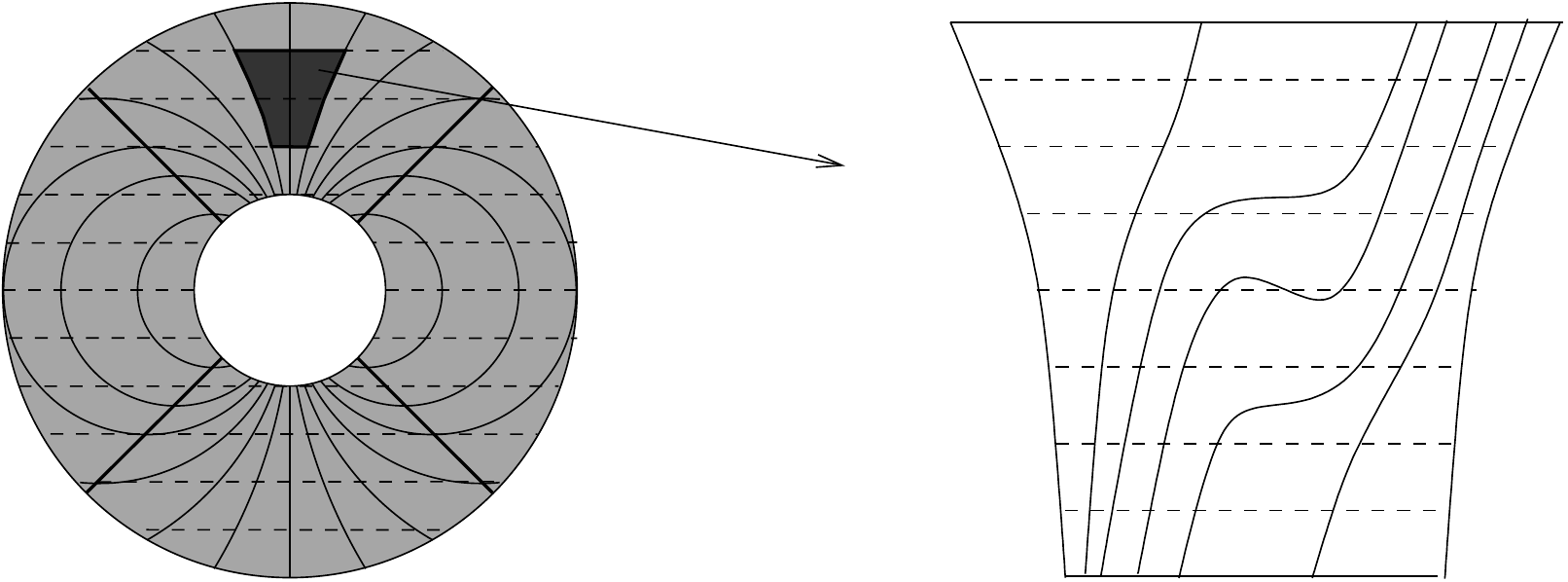}
\caption{Dot-lines represents the stable foliation and curved lines are the unstable foliation. 
In this way $f$ is not 2-expansive but it is 
$C^3$-robustly 3-expansive.}
\label{fig3exp}
\end{center}
\end{figure}
There are two curves that are topologically transverse but there is a tangency of order 2. 
Between these two curves there are points of non topologically transversality.
The unstable arcs are modeled by the one parameter family 
$$p_a(x)=x^3+(a^2-1)x+9a$$ for $x,a\in[-2,2]$. 
The presence of the term $9a$ implies that $\partial p_a/\partial a>0$ for all $x,a\in[-2,2]$.
If $|a|>1$ then $p'_a(x)>0$ for all $x$, so we have transversality. 
If $|a|=1$ then $p_a(x)$ is a translation of $x^3$, so there is a tangency of order 2 at $x=0$. 
If $|a|<1$ then $p_a(x)$ has a local maximum and local minimum that are close if $|a|$ is close to 1. 
Therefore, we see that $f$ is not 2-expansive. 
It is $C^3$-robustly 3-expansive because $p'''_a(x)=6\neq 0$ for all $x,a\in[-2,2]$.

For the case $r=4$ we consider an open set $U\subset A$ containing a cuadratic tangency as in Figure \ref{fig4exp}. 
The map $\psi$ is changed in $U$ in such a way that the unstable arcs corresponds to the curves on the right hand 
of the figure. They can be modeled with the polynomials $$p_a(x)=x^4+(a^2-1)x^2+16a.$$

\begin{figure}[h]
\begin{center}
\includegraphics[scale=.5]{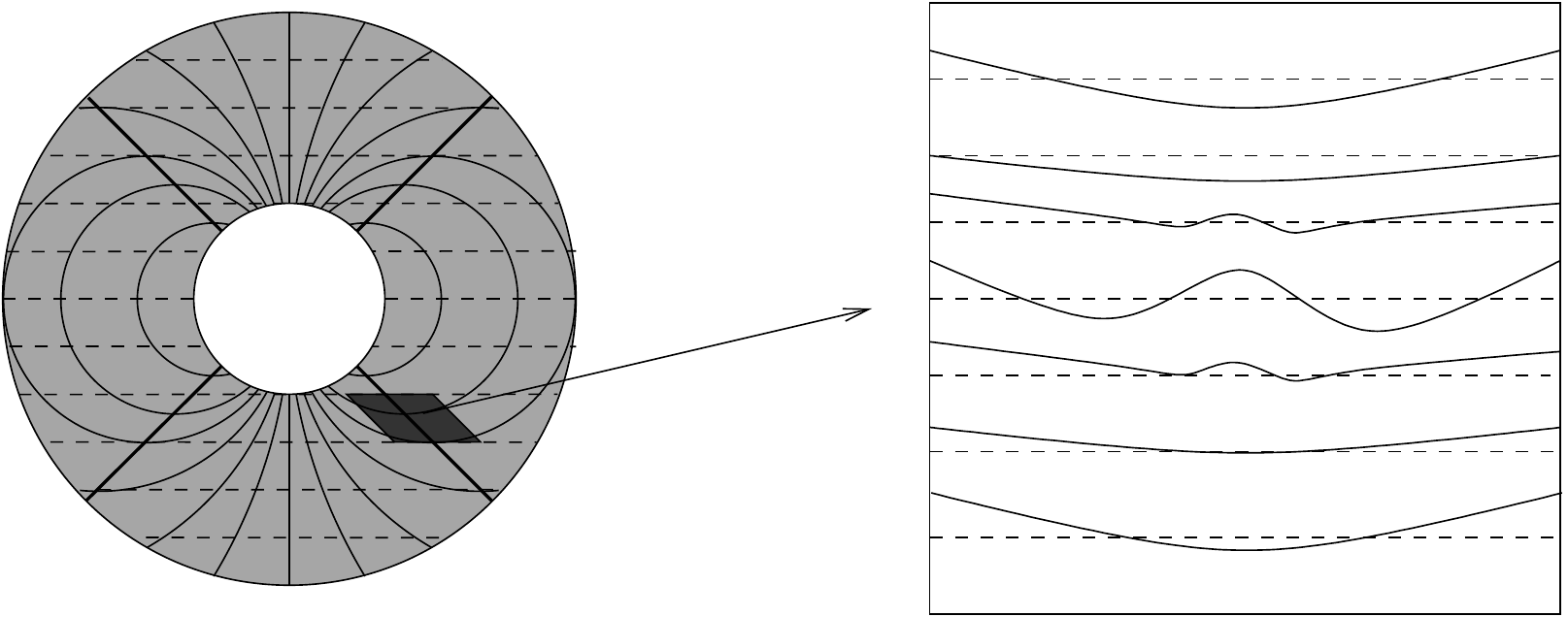}
\caption{Stable and unstable foliations for a 
$C^4$-robustly 4-expansive diffeomorphism that is not 3-expansive.}
\label{fig4exp}
\end{center}
\end{figure}

The general case $r\geq 5$ follows the same ideas.
For $r$ odd $\psi$ is changed near a box of local product structure. 
For $r$ even $\psi$ is changed near a quadratic tangency. 
\end{proof}

In particular we have:

\begin{cor}
  There are $C^r$-robustly cw-expansive surface diffeomorphisms that are not Anosov.
\end{cor}

\end{document}